\documentclass{compositio}
\usepackage{amsmath}
\usepackage{hyperref}

\newtheorem{thm}{Theorem}[section]
\newtheorem{lem}[thm]{Lemma}
\newtheorem{cor}[thm]{Corollary}
\newtheorem{prop}[thm]{Proposition}
\newtheorem{conj}[thm]{Conjecture}

\theoremstyle{definition}
\newtheorem{rem}[thm]{Remark}
\newtheorem{rems}[thm]{Remarks}
\newtheorem{defn}[thm]{Definition}

\renewcommand{\theenumi}{(\roman{enumi})}

\def\Z{\mathbf{Z}}
\def\Q{\mathbf{Q}}

\def\R{\mathbf{R}}
\def\C{\mathbf{C}}

\def\Y{\mathbf{Y}}
\def\bN{\mathbf{N}}

\def\Zp{\Z_p}
\def\Zl{\Z_\ell}

\def\Ql{\Q_ \ell}

\def\O{\mathcal{O}}
\def\E{\mathcal{E}}
\def\I{\mathcal{I}^{\mathrm{new}}}
\def\J{\mathcal{I}^{\mathrm{old}}}
\def\N{\mathcal{N}}

\def\cF{\mathcal{F}}
\def\D{\mathcal{D}}
\def\T{\mathcal{T}}

\def\ld{\mathcal{h}}
\def\rd{\mathcal{i}}

\def\p{\mathfrak{p}}

\def\Gal{\mathrm{Gal}}
\def\ord{\mathrm{ord}}
\def\tr{\mathrm{tr}}
\def\f{\mathrm{f}}

\def\Pic{\mathrm{Pic}}
\def\sign{\mathrm{sign}}
\def\Frob{\mathrm{Fr}}

\def\too{\longrightarrow}
\def\map#1{\;\xrightarrow{#1}\;}
\def\isom{\xrightarrow{\sim}}
\def\hookto{\hookrightarrow}
\def\onto{\twoheadrightarrow}

\def\bmu{\boldsymbol{\mu}}
\def\phifs{\phi^{\mathrm{fs}}}
\def\ksys{\boldsymbol{\kappa}}
\def\projI#1#2{{\ld#2\rd_{#1}^{\mathrm{new}}}}
\def\bR#1#2{S_{#1,#2}}
\def\symm{\mathfrak{S}}
\def\pile#1#2{\genfrac{}{}{0pt}{1}{#1}{#2}}

\def\gr#1#2{[\;\cdot\;]_{#1}^{#2}}
\def\grx#1#2#3{[#3]_{#1}^{#2}}
\def\In{\I_n}
\def\Jn{\J_n}
\def\KS{\mathbf{KS}}
\def\preKS{\mathbf{PKS}}

\title{Refined class number formulas and Kolyvagin systems}

\author{Barry Mazur}
\email{mazur@math.harvard.edu}
\address{Department of Mathematics, 
Harvard University,
Cambridge, MA 02138, 
USA}
\author{Karl Rubin}
\email{krubin@math.uci.edu}
\address{Department of Mathematics, 
UC Irvine,
Irvine, CA 92697, 
USA}

\thanks{This material is based upon work supported by the 
National Science Foundation under grants DMS-0700580 and DMS-0757807.}

\begin{document}


\begin{abstract}
We use the theory of Kolyvagin systems to prove (most of) a 
refined class number formula conjectured by Darmon.  We show that for every 
odd prime $p$, each side of Darmon's conjectured formula (indexed by 
positive integers $n$) is ``almost'' a $p$-adic Kolyvagin system as $n$ 
varies.  Using the fact that the space of Kolyvagin systems is free of rank 
one over $\mathbf{Z}_p$, we show that Darmon's formula for arbitrary $n$ 
follows from the case $n=1$, which in turn follows from classical formulas.
\end{abstract}

\maketitle

\section{Introduction}

In this paper we use the theory of Kolyvagin systems to 
prove (most of) a conjecture of Darmon from \cite{darmon}.

In \cite[Conjecture 4.1]{gross}, inspired by work of the first author and 
Tate \cite{mazurtate}, and of Hayes \cite{hayes}, Gross conjectured a 
``refined class number formula'' for abelian extensions $K/k$ of global fields.  
Attached to this extension (and some chosen auxiliary data) 
there is a generalized Stickelberger element $\theta_{K/k} \in \Z[G]$, where $G := \Gal(K/k)$, 
with the property that for every complex-valued character $\chi$ of $G$, 
$\chi(\theta_{K/k})$ is essentially $L(K/k,\chi,0)$ (modified by the chosen auxiliary data). 
Gross' conjectural formula is a congruence for $\theta_{K/k}$, modulo a certain specified 
power of the augmentation ideal of $\Z[G]$, in terms of a regulator that Gross defined.

In a very special case, Darmon formulated an analogue of Gross' conjecture involving 
first derivatives of $L$-functions at $s=0$.
Suppose $F$ is a real quadratic field, and $K_n := F(\bmu_n)$ 
is the extension of $F$ generated by $n$-th roots of unity, 
with $n$ prime to the conductor of $F/\Q$.  
Darmon defined a Stickelberger-type 
element $\theta_n' \in K_n^\times \otimes \Z[\Gal(K_n/F)]$, interpolating 
the first derivatives $L'(\chi\omega_F,0)$, where $\omega_F$ is 
the quadratic character attached to $F/\Q$ and $\chi$ runs through even Dirichlet characters 
of conductor $n$.  
Darmon conjectured that $\theta_n'$ is congruent, modulo a specified 
power of the augmentation ideal, to a regulator that he defined.
See \S\ref{statement} and Conjecture \ref{dconj} below for a precise statement.

Our main result is a proof of Darmon's conjecture ``away from the $2$-part''.  
In other words, we prove that the difference of the two sides of Darmon's 
conjectured congruence is an element of $2$-power order.  

The idea of our proof is a simple application of the results proven 
in \cite{kolysys}.  For every odd prime $p$ we show that although 
neither the left-hand side nor the right-hand of Darmon's 
conjectured congruence (as $n$ varies) is a ``Kolyvagin system'' as defined in \cite{kolysys}, 
each side is {\em almost} a Kolyvagin system; moreover, both sides fail to be 
Kolyvagin systems in precisely the same way. That is, we show that the left-hand side 
and right-hand side form what we call in this paper {\em pre-Kolyvagin systems} 
in the sense that they each  satisfy the specific set of local and global 
compatibility relations given in Definition \ref{preksdef} below.  
It seems that pre-Kolyvagin systems are what tend to occur ``in nature'', 
while Kolyvagin systems satisfy a cleaner set of axioms.  
We show that the two concepts are equivalent, 
by constructing (see Proposition \ref{preksks}) a natural transformation $\T$ that turns 
pre-Kolyvagin systems into Kolyvagin systems and has the properties that:
\begin{itemize}
\item  $\T$ does not change the term  associated to $n=1$, and
\item $\T$ is an isomorphism from the $\Zp$-module of pre-Kolyvagin systems 
to the $\Zp$-module of Kolyvagin systems.
\end{itemize}
Since it was proved in \cite{kolysys} that (in this situation) the 
space of Kolyvagin systems is a free $\Zp$-module of rank one, 
it follows that if two pre-Kolyvagin systems agree when $n=1$, 
then they agree for every $n$.  In the case $n= 1$, Darmon's congruence 
follows from classical formulas for $L'(\omega_F,0)$, so 
we deduce that (the $p$-part, for every odd prime $p$ of) 
Darmon's conjectured congruence formula holds for all $n$.

Darmon's conjecture begs for a generalization.  
A naive generalization, even just to the case where 
$F$ is a real abelian extension of $\Q$, is unsuccessful 
because the definition of Darmon's regulator 
does not extend to the case where $[F:\Q] > 2$.  
In a forthcoming paper we will use the ideas and conjectures of 
\cite{RS} to show how both Gross' and Darmon's conjectures are special cases of a much more 
general conjecture.  In the current paper we treat only Darmon's conjecture because it can be 
presented and proved in a very concrete and explicit manner.

The paper is organized as follows.  In \S\ref{nota} we describe our setting and notation, and 
in \S\ref{statement} we state Darmon's conjecture and our main result (Theorem \ref{mainthm}).  
In \S\ref{augq} we recall some work of Hales \cite{hales} 
on quotients of powers of augmentation ideals, that will enable us to translate the 
definition of Kolyvagin system given in \cite{kolysys} into a form that will be more 
useful for our purposes here.  In \S\ref{kssect} we give the definition of a Kolyvagin system 
(for the Galois representation $\Zp(1) \otimes \omega_F$).  
In \S\ref{prekssect} we define pre-Kolyvagin system, and give an isomorphism between the 
space of pre-Kolyvagin systems and the space of Kolyvagin systems.
In \S\ref{cupks} (resp., \S\ref{rpks}) we show that the ``Stickelberger'' side 
(resp., regulator side) of Darmon's formula is a pre-Kolyvagin system as $n$ varies.  
Finally, in \S\ref{pfsect} we combine the results of the previous sections to 
prove Theorem \ref{mainthm}.

\section{Setting and notation}
\label{nota}
 
Fix once and for all a real quadratic field $F$, and let $f$ 
be the conductor of $F/\Q$.  Let $\omega = \omega_F$ be the quadratic Dirichlet 
character associated to $F/\Q$, and 
$\tau$ the nontrivial element of $\Gal(F/\Q)$.
If $M$ is a $\Gal(F/\Q)$-module, we let $M^-$ be the subgroup of elements of 
$M$ on which $\tau$ acts as $-1$.

Throughout this paper $\ell$ will always denote a prime number.
Let $\N$ denote the set of squarefree positive integers prime to $f$.  If 
$n \in\N$ let $n_+$ be the product of all primes dividing $n$ that split in $F/\Q$, 
and $r(n) \in \Z_{\ge 0}$ the number of prime divisors of $n_+$:
\begin{align*}
n_+ &:= \prod_{{\ell\mid n},{\omega(\ell) = 1}} \ell, \\ 
   r(n) &:= \#\{\ell : \ell \mid n_+\} = \#\{\ell : \text{$\ell \mid n$ and $\ell$ splits in $F$}\}.
\end{align*}
For every $n \in \N$ let $\bmu_n$ be the 
Galois module of $n$-th roots of unity in $\bar{\Q}$, define
$$
\Gamma_n := \Gal(F(\bmu_n)/F) \cong \Gal(\Q(\bmu_n)/\Q) \cong (\Z/n\Z)^\times,
$$
and let $I_n$ denote the augmentation ideal of $\Z[\Gamma_n]$, which is generated over 
$\Z$ by $\{\gamma-1 : \gamma \in \Gamma_n\}$. 
There is a natural isomorphism
\begin{equation}
\label{caug}
\Gamma_n \cong I_n/I_n^2
\end{equation}
defined by sending $\gamma \in \Gamma_n$ to $\gamma-1 \pmod{I_n^2}$.
If $m \mid n$ then we can view $\Gamma_m$ either as the quotient $\Gal(F(\bmu_m)/F)$ 
of $\Gamma_n$, or as the subgroup $\Gal(F(\bmu_n)/F(\bmu_{n/m}))$.  With the latter 
identification we have
$$
\Gamma_n = \prod_{\ell \mid n} \Gamma_\ell, \quad I_n/I_n^2 = \bigoplus_{\ell\mid n}I_\ell/I_\ell^2
$$
the product and the sum taken over primes $\ell$ dividing $n$. 

We will usually write the group operation in multiplicative groups such as $F^\times$ 
with standard multiplicative notation (for example, with identity element $1$).  However, 
when dealing with ``mixed'' groups such as $F^\times \otimes I_n^r/I_n^{r+1}$, 
we will write the operation additively and use $0$ for the identity element.

Fix an embedding $\bar\Q \hookto \C$.

\section{Statement of the conjecture}
\label{statement}

In this section we state our modified version of Darmon's conjecture 
(mostly following \cite{darmon}) and our main result (Theorem \ref{mainthm}).

If $n \in \N$, let $\zeta_n \in \bmu_n$ be the 
inverse image of $e^{2 \pi i /n}$ under the chosen embedding $\bar\Q \hookto \C$, 
and define the cyclotomic unit
$$
\alpha_n := 
   \prod_{\gamma \in \Gal(\Q(\bmu_{nf})/\Q(\bmu_n))}\gamma(\zeta_{nf}-1)^{\omega_F(\gamma)}
   \quad\in~ F(\bmu_n)^\times
$$
and the ``first derivative $\theta$-element''
$$
\theta'_n = 
   \sum_{\gamma \in \Gamma_n}\gamma(\alpha_n) \otimes \gamma
   \quad\in~ F(\bmu_n)^\times \otimes \Z[\Gamma_n].
$$

\begin{rem}
The element $\theta'_n$ is an ``$L$-function derivative evaluator'' in the sense that 
for every even character $\chi : \Gamma_n \to \C^\times$, classical formulas give
$$
(\log|\cdot| \otimes \chi)(\theta'_n) 
   := \sum_{\gamma\in\Gamma_n}\chi(\gamma)\log|\gamma(\alpha_n)| = -2L'_n(0,\omega_F\chi)
$$
where $L_n(s,\omega_F\chi)$ is the Dirichlet $L$-function with Euler factors 
at primes dividing $n$ removed, and $|\cdot|$ is the absolute value corresponding to 
our chosen embedding $\bar{\Q} \hookto \C$.
\end{rem}

Suppose $n \in \N$.  
Let $X_n$ be the group of divisors of $F$ supported above $n\infty$, and 
let $\E_n := \O_F[1/n]^\times$, the group of $n$-units of $F$. 
We will write the action of $\Z[\Gamma_n]$ on $\E_n$ additively, 
so in particular $(1-\tau)\E_n = \{\epsilon/\epsilon^\tau : \epsilon \in \E_n\}$.

Let $\lambda_0 \in X_n$ be the archimedean place of $F$ corresponding 
to our chosen embedding $\bar\Q \hookto \C$.

\begin{lem}
\label{freer+1}
Suppose $n \in \N$, and let $r = r(n)$.
\begin{enumerate}
\item
We have
$X_{n}^- = X_{n_+}^-$, $\E_{n}^- = \E_{n_+}^-$, and 
$(1-\tau)\E_{n} = (1-\tau)\E_{n_+}$.
\item
The group $(1-\tau)\E_n$ is a free abelian group of rank $r+1$, 
and is a subgroup of finite index in $\E_n^-$.
\item
The group $X_n^-$ is a free abelian group of rank $r+1$.  If 
$n_+ = \prod_{i=1}^r \ell_i$, and $\ell_i = \lambda_i\lambda_i^\tau$, then 
$\{\lambda_0-\lambda_0^\tau,\lambda_1-\lambda_1^\tau,\ldots,\lambda_r-\lambda_r^\tau,\}$ 
is a basis of $X_n^-$.
\end{enumerate}
\end{lem}

\begin{proof}
The only part that is not clear is that $(1-\tau)\E_n$ is torsion-free, i.e., 
$-1 \notin (1-\tau)\E_n$.  Let $d > 1$ be a squarefree integer such that 
$F = \Q(\sqrt{d})$. If $x^\tau = -x$, then $x/\sqrt{d} \in \Q$, so $x$ 
is not a unit at the primes dividing $d$.  Since $n$ is 
prime to $d$, we cannot have $x \in \E_n$.
\end{proof}

\begin{defn}
A {\em standard} $\Z$-basis of $X_n^-$ is a basis of the form described 
in Lemma \ref{freer+1}(iii).  
Given a standard basis of $X_n^-$, a $\Z$-basis $\{\epsilon_0, \ldots, \epsilon_r\}$ 
of $(1-\tau)\E_n$ will be called {\em oriented} 
if the (regulator) determinant of the logarithmic embedding
$$
(1-\tau)\E_n \too X_n^- \otimes \R, \quad \epsilon \mapsto 
  \sum_{\lambda\mid n_+\infty}\log|\epsilon|_{\lambda}\cdot\lambda
$$
with respect to the two bases is positive.  Concretely, this regulator 
is the determinant of the matrix whose entry in row $i$ and column $j$ is
$\log|\epsilon_j|_{\lambda_i}$.
\end{defn}

\begin{rem}
Choosing a standard basis of $X_n^-$ is equivalent to ordering the prime 
divisors $\ell_i$ of $n_+$ and choosing one prime of $F$ above each $\ell_i$.

Any basis of $(1-\tau)\E_n$ can be oriented either by reordering the 
basis, or inverting one of the basis elements.
\end{rem}

\begin{defn}
Suppose $n \in \N$ and $\lambda$ is a prime of $F$ dividing $n_+$.  
Define a homomorphism 
$$
\gr{\lambda}{n} : F^\times \too I_n/I_n^2
$$
by
$$
\grx{\lambda}{n}{x} = [x, F_\lambda(\bmu_n)/F_\lambda]-1 \pmod{I_n^2}
$$
where $[x, F_\lambda(\bmu_n)/F_\lambda] \in \Gamma_n$ is the local Artin symbol.  
Note that if $\ord_\lambda(x) = 0$, then $[x, F_\lambda(\bmu_n)/F_\lambda]$ 
belongs to the inertia group $\Gamma_\ell \subset \Gamma_n$, so 
$\grx{\lambda}{n}{x} = \grx{\lambda}{\ell}{x} \in I_\ell/I_\ell^2$ 
and $\grx{\lambda}{n/\ell}{x} = 0$.  In general, if $d \mid n$ then 
$$\grx{\lambda}{n}{x} = \grx{\lambda}{d}{x} + \grx{\lambda}{n/d}{x} 
   \in I_d/I_d^2 \oplus I_{n/d}/I_{n/d}^2 = I_n/I_n^2.$$
\end{defn}

\begin{defn}
\label{Rdef}
(See \cite[p.\ 308]{darmon}.)
Suppose $n \in \N$, and let $r = r(n)$.
Choose a standard basis $\{\lambda_0-\lambda_0^\tau,\ldots,\lambda_r-\lambda_r^\tau\}$ 
of $X_n^-$ and an oriented basis 
$\{\epsilon_0, \ldots, \epsilon_r\}$ of $(1-\tau)\E_n$, and define 
the regulator $R_n \in \E_n^- \otimes I_n^r/I_n^{r+1}$ by 
$$
R_n := \left|
\begin{array}{ccccccc}
\epsilon_0 & \epsilon_1 & \cdots & \epsilon_r \\
\grx{\lambda_1}{n}{\epsilon_0} & \grx{\lambda_1}{n}{\epsilon_1} & \cdots & \grx{\lambda_1}{n}{\epsilon_r} \\
\vdots & \vdots && \vdots\\
\grx{\lambda_r}{n}{\epsilon_0} & \grx{\lambda_r}{n}{\epsilon_1} & \cdots & \grx{\lambda_r}{n}{\epsilon_r}
\end{array}
\right| \quad  \in (1-\tau)\E_n \otimes I_n^r/I_n^{r+1}.
$$
This determinant, and the ones that follow below, are meant to be evaluated 
by expanding by minors along the top row, i.e.,
\begin{equation}
\label{minors}
R_n := \sum_{j=0}^r (-1)^{j}\epsilon_j \otimes \det(A_{1j})
\end{equation}
where $A_{1j}$ is the $r \times r$ matrix (with entries in $I_n/I_n^2$) 
obtained by removing the first row and 
$j$-th column of the matrix above.
\end{defn}

Note that this definition of $R_n$ does not depend on the choice of $\Z$-bases.  
The possible ambiguity of $\pm1$ is removed by requiring that the basis 
of $(1-\tau)\E_n$ be oriented.

Let $h_n$ denote the ``$n$-class number'' of $F$, i.e., the order of the ideal class 
group $\Pic(\O_F[1/n])$.  For the rest of this section we write simply $r$ instead of $r(n)$.

\begin{thm}[(Darmon {\cite[Theorem 4.2]{darmon}})]
For every $n \in \N$, we have 
$$\theta'_n \in F(\bmu_n)^\times \otimes I_n^r.$$
\end{thm}

For $n \in \N$, let $\tilde\theta'_n$ denote the image of $\theta'_n$ 
in $F(\bmu_n)^\times \otimes I_n^r/I_n^{r+1}$.  Let $s$ be the number of prime divisors 
of $n/n_+$; we continue to denote by $r$ the number of prime factors of $n_+$.

The following is a slightly modified version of Darmon's ``leading term'' conjecture 
\cite[Conjecture 4.3]{darmon}.

\begin{conj}
\label{dconj}
For every $n \in \N$, we have
$$
\tilde\theta'_n = -2^s h_n R_n \quad 
   \text{in $(F(\bmu_n)^\times/\{\pm1\}) \otimes I_n^r/I_n^{r+1}$}.
$$
\end{conj}

The main theorem of this paper is the following.

\begin{thm}
\label{mainthm}
For every $n \in \N$, we have
$$
\tilde\theta'_n = -2^s h_n R_n \quad 
   \text{in $F(\bmu_n)^\times \otimes I_n^r/I_n^{r+1} \otimes \Z[1/2]$}.
$$
\end{thm}

In other words, the $p$-part of Conjecture \ref{dconj} holds for every odd prime 
$p$; in still other words, $\tilde\theta'_n + 2^s h_n R_n$ has $2$-power order 
in $F(\bmu_n)^\times \otimes I_n^r/I_n^{r+1}$.

A key step in the proof of Theorem \ref{mainthm} is the following observation.

\begin{prop}[(Darmon {\cite[Theorem 4.5(1)]{darmon}})]
\label{n=1}
Conjecture \ref{dconj} holds if $n=1$.
\end{prop}

\begin{proof}
When $n=1$ we have $r=0$, $I_n^r/I_n^{r+1} = \Z$, 
$
\tilde\theta'_1 = \theta'_1 = \alpha_1 \in \O_F^\times,
$
and $R_1 = \epsilon/\epsilon^\tau$, where $\epsilon$ is a generator of 
$\O_F^\times/\{\pm1\}$ and $|\epsilon/\epsilon^\tau| = |\epsilon|^2 > 1$ 
at our specified archimedean place.  
Dirichlet's analytic class number formula shows that 
$$
-\frac{1}{2}\log|\alpha_1| = L'(0,\omega_F) = h_F\log|\epsilon| 
   = \frac{1}{2}h_F \log|\epsilon/\epsilon^\tau |
$$
where $h_F = h_1$ is the class number of $F$. 
Hence $\alpha_1 = \pm(\epsilon/\epsilon^\tau)^{-h_F}$ in $\O_F^\times$. 
\end{proof}

\begin{rems}
(i) In Darmon's formulation \cite[Conjecture 4.3]{darmon}, 
the regulator $R_n$ was defined with respect to a basis of $\E_n^-/\{\pm1\}$ 
instead of $(1-\tau)\E_n$, and 
there was an extra factor of $2$ on the right-hand side.  
This agrees with Conjecture \ref{dconj} if and only if $[\E_n^-:\pm(1-\tau)\E_n] = 2$, 
i.e., if and only if $-1 \notin \bN_{F/\Q}\E_n$.

(ii) The ambiguity of $\pm1$ in Conjecture \ref{dconj} is necessary.  
Namely, even when $n=1$, 
we may only have $\tilde\theta'_1 = h_1R_1$ in $\O_F^\times/\{\pm1\}$.  
Since $\alpha_1$ is always positive (it is a norm from a CM field to $F$), 
the proof of Proposition \ref{n=1} shows that $\tilde\theta'_1 \ne -h_1R_1$ 
in $F^\times$ when $h_F$ is odd and $\O_F^\times$ has a unit of norm $-1$.  
Note that in this case $\tilde\theta'_1$ and $-h_1R_1$ differ (multiplicatively) 
by an element 
of order $2$ in $F^\times$, so the discrepancy disappears when we tensor with $\Z[1/2]$.  
\end{rems}

\section{Augmentation quotients}
\label{augq}

\begin{defn}
Suppose $n \in \N$, and let $r = r(n)$.   
Let $\In \subset I_n^r/I_n^{r+1}$ be the (cyclic) subgroup generated by monomials 
$\prod_{\ell\mid n_+}(\gamma_\ell-1)$ with $\gamma_\ell \in \Gamma_\ell$. 
Let $\J_n \subset I_n^r/I_n^{r+1}$ be the subgroup generated by monomials 
$\prod_{i=1}^r(\gamma_i - 1)$ where each $\gamma_i \in \Gamma_{\ell_i}$ for 
some $\ell_i$ dividing $n$, and $\{\ell_1,\ldots,\ell_r\} \ne \{\ell: \ell \mid n_+\}$ 
(i.e., either one of the $\ell_i$ divides $n/n_+$, or $\ell_i = \ell_j$ for some $i \ne j$).  
If $n = d_1d_2$ then there is a natural identification 
$\In = \I_{d_1}\I_{d_2} \subset I_{n}^{r}/I_{n}^{r+1}$, and if $n = \ell$ is prime then 
$\I_\ell = I_\ell/I_\ell^2$ and $\J_\ell = 0$.
\end{defn}

If $d \mid n$, let 
$$
\pi_d : \Z[\Gamma_n] \onto \Z[\Gamma_d] \hookto \Z[\Gamma_n]
$$ 
denote the composition of the natural maps.  
We also write $\pi_d$ for the induced map on $I_n^k/I_n^{k+1}$ for $k \ge 0$. 

The following proposition is based on work of Hales \cite{hales}. 

\begin{prop}
\label{crit}
Suppose $n \in\N$, and $r=r(n)$.
Then:
\begin{enumerate}
\item
$I_n^r/I_n^{r+1} = \In \oplus \Jn$.
\item
If $d \mid n_+$ and $d > 1$, then $\pi_{n/d}(\In) = 0$
and $\pi_{n/d}(I_n^r/I_n^{r+1}) \subset \Jn$.
\item
$\In = \{v \in I_n^r/I_n^{r+1} : \text{$\pi_{n/\ell}(v) = 0$ for every $\ell$ dividing $n_+$}\}$.
\item
The map $\otimes_{\ell\mid n_+}\Gamma_\ell \to \In$ defined by 
$\otimes_{\ell\mid n_+} \gamma_\ell \mapsto \prod_{\ell\mid n_+}(\gamma_\ell-1)$ 
is an isomorphism.
\end{enumerate} 
\end{prop}

\begin{proof}
Let $A_n$ be the polynomial ring $\Z[Y_\ell : \ell \mid n]$ with one variable 
$Y_\ell$ for each prime $\ell$ dividing $n$.  Fix a generator $\sigma_\ell$ of 
$\Gamma_\ell$ for every $\ell$ dividing $n$, and define a map $A_n \to \Z[\Gamma_n]$ 
by sending $Y_\ell \mapsto \sigma_\ell-1$. 
By Corollary 2 of \cite{hales}, this map induces an isomorphism from the 
homogeneous degree-$r$ part of $A_n/(J_n+J_n')$ to 
$I_n^r/I_n^{r+1}$, where $J_n$ is the ideal of $A_n$ generated by 
$\{(\ell-1)Y_\ell : \ell \mid n\}$, and $J_n'$ is the ideal generated by certain other explicit 
homogeneous relations (see \cite[Lemma 2]{hales}).  
The only fact we need about these ``extra'' relations is:
\begin{equation}
\label{star}
\text{\em if $f \in J_n'$, then every monomial that occurs in $f$ is divisible by 
the square of some $Y_\ell$.}
\end{equation}
Note that $\In$ is the image in $I_n^r/I_n^{r+1}$ of the subgroup of 
$A_n/(J_n+J_n')$ generated by $\Y_n$, where $\Y_n := \prod_{\ell\mid n_+}Y_\ell$.
Similarly, $\Jn$ is the image of the subgroup generated by all other 
monomials of degree $r$.  
By \eqref{star}, $\Y_n$ does not occur in any of the relations in $J_n'$, 
and assertion (i) follows.

Assertion (ii) is clear, since $\pi_{n/d}$ kills those monomials 
that include $(\gamma - 1)$ with $\gamma \in \Gamma_\ell$ for $\ell$ dividing $d$, 
and leaves the other monomials unchanged.

Fix $v \in I_n^r/I_n^{r+1}$.  
If $v \in \In$ and $\ell \mid n_+$, then $\pi_{n/\ell}(v) = 0$ by (ii).  
Conversely, suppose that $\pi_{n/\ell}(v) = 0$ 
for every $\ell$ dividing $n_+$.  Choose $f \in A_n$ homogeneous of degree $r$ 
representing $v$, and suppose $f$ has the minimum number of monomials among 
all representatives of $v$.  We will show that $\Y_n \mid f$, and hence 
$v \in \In$. 

Fix a prime $\ell$ dividing $n_+$.  The map 
$\pi_{n/\ell}: \Z[\Gamma_n] \onto \Z[\Gamma_{n/\ell}] \hookto \Z[\Gamma_n]$ corresponds to 
the map $A_n \to A_n$ defined by setting $Y_\ell = 0$.
Since $\pi_{n/\ell}(v) = 0$, substituting $Y_\ell = 0$ in $f$ gives a 
relation in $J_n+J_n'$, i.e., $f = Y_\ell\cdot g + h$ where $g$ is 
homogeneous of degree $r-1$, $h \in J_n+J_n'$, and $Y_\ell$ does not occur in $h$.  
But then $Y_\ell\cdot g$ represents $v$, so the minimality assumption on $f$ 
implies that $h = 0$.  Therefore $Y_\ell \mid f$ for every $\ell$ dividing $n_+$, 
so $\Y_n\mid f$ and $v \in \In$.  This proves (iii).

Let $g := \gcd(\{\ell-1 : \ell\mid n_+\})$.  
Then $g \Y_n \in J_n$.
It follows from \eqref{star} that the monomial $\Y_n$ only 
occurs in elements of $J_n+J_n'$ with coefficients divisible by $g$.  
Therefore $\In$ is cyclic of order $g$, and so is $\otimes_{\ell\mid n_+}\Gamma_\ell$.
Clearly the map $\otimes_{\ell\mid n_+}\Gamma_\ell \to \In$ of (iv) is surjective, 
so it must be an isomorphism.
\end{proof}

If $v \in I_n^r/I_n^{r+1}$, let $\projI{n}{v}$ denote the projection of $v$ to 
$\In$ under the splitting of Proposition \ref{crit}(i).  
We will use the following lemma without explicit reference in some of our computations
in \S\ref{prekssect} and \S\ref{rpks}.  Its proof is left as an exercise.

\begin{lem}
Suppose $d\mid n$, $v \in \I_{n/d}$, and $w \in I_n^{r(d)}/I_n^{r(d)+1}$.  Then 
$$
\projI{n}{vw} = \projI{n}{v\pi_d(w)} = v \projI{d}{\pi_d(w)}.
$$
\end{lem}

\section{Kolyvagin systems}
\label{kssect}

Fix an odd prime $p$.  To prove Theorem \ref{mainthm} we need to 
introduce Kolyvagin systems, as defined in \cite{kolysys}.  (See in particular 
\cite[\S6.1]{kolysys}, and also \cite{babykolysys}, for the case of Kolyvagin systems 
associated to even Dirichlet characters that we use here.) 

Let $\hat{F}^\times$ denote the $p$-adic completion of $F^\times$.  
Similarly, for every rational prime $\ell$ 
let $F_\ell := F \otimes \Ql$, $\O_\ell := \O_F \otimes \Zl$, 
and define $\hat{F}^\times_\ell$ and $\hat{\O}^\times_\ell$
to be their $p$-adic completions.
We define the ``finite subgroup'' $\hat{F}^\times_{\ell,\f}$ 
to be the ``unit part'' of $\hat{F}^\times_\ell$
$$
\hat{F}^\times_{\ell,\f} := \hat{\O}_\ell^\times \subset \hat{F}^\times_\ell.
$$

If $\ell = \lambda{\lambda^\tau}$ splits in $F$,
define the ``transverse subgroup'' 
$\hat{F}^\times_{\ell,\tr} \subset \hat{F}^\times_\ell$ to be the 
(closed) subgroup generated by $(\ell,1)$ and $(1,\ell)$, where we identify $F_\ell^\times$ 
with $F_\lambda^\times \times F_{\lambda^\tau}^\times \cong \Ql^\times \times \Ql^\times$.  
Then we have a canonical 
splitting 
$\hat{F}^\times_\ell = \hat{F}^\times_{\ell,\f} \times \hat{F}^\times_{\ell,\tr}$, 
and since $p$ is odd
\begin{equation}
\label{ftrs}
(\hat{F}^\times_\ell)^- 
   = (\hat{F}^\times_{\ell,\f})^- \times (\hat{F}^\times_{\ell,\tr})^-.
\end{equation}

\begin{defn}
\label{phifsdef}
If $\ell \ne p$ splits in $F$, define the {\em finite-singular isomorphism}
$$
\phifs_\ell : (\hat{F}^\times_{\ell,\f})^- \isom (\hat{F}^\times_{\ell,\tr})^- \otimes I_\ell/I_\ell^2
$$
by 
\begin{align*}
\phifs_\ell(x) &= (\ell,1) \otimes ([x_\lambda,F_\lambda(\bmu_\ell)/F_\lambda]-1) 
   + (1,\ell) \otimes ([x_{\lambda^\tau},F_{\lambda^\tau}(\bmu_\ell)/F_{\lambda \tau}]-1) \\
   &= (\ell,\ell^{-1}) \otimes ([x_\lambda,F_\lambda(\bmu_\ell)/F_\lambda]-1)
\end{align*}
where 
$x = (x_\lambda,x_{\lambda^\tau}) 
   \in \hat{F}_\lambda^\times \times \hat{F}_{\lambda^\tau}^\times
   = \hat{\Q}_\ell^\times \times \hat{\Q}_\ell^\times$ 
   with $x_{\lambda^\tau} = x_\lambda^{-1} \in \hat{\Z}_\ell^\times$, 
and $[\;\cdot\;,F_\lambda(\bmu_\ell)/F_\lambda]$ is the local Artin symbol.
(Concretely, note that if $u \in \Zl^\times$ then $[u,F_\lambda(\bmu_\ell)/F_\lambda]$ 
is the automorphism in $\Gamma_\ell$ that sends $\zeta_\ell$ to $\zeta_\ell^{u^{-1}}$.)
Then $\phifs_\ell$ is a well-defined isomorphism 
(both the domain and range are free of rank one over $\Zp/(\ell-1)\Zp$), 
independent of the choice of $\lambda$ versus $\lambda^\tau$.
\end{defn}

\begin{defn}
\label{ksdef}
Let $\N_p := \{n \in \N : p \nmid n\}$.  
A {\em Kolyvagin system} $\ksys$ 
(for the Galois representation $\Zp(1) \otimes \omega_F$) is a collection
$$
\{\kappa_n \in (\hat{F}^\times)^- \otimes \In : n \in \N_p\}
$$
satisfying the following properties for every rational prime $\ell$.  
Let $(\kappa_n)_\ell$ denote the image of $\kappa_n$ in 
$(\hat{F}^\times_\ell)^- \otimes \In$.
\begin{enumerate}
\item
If $\ell \nmid n$, then 
$(\kappa_n)_\ell \in (\hat{F}^\times_{\ell,\f})^- \otimes \In$.
\item
If $\ell \mid n_+$, then
$(\kappa_n)_\ell = (\phifs_\ell \otimes 1)(\kappa_{n/\ell,\ell})$.
\item
If $\ell \mid n/n_+$, then $\kappa_n = \kappa_{n/\ell}$.
\end{enumerate}
Let $\KS(F)$ denote the $\Zp$-module of Kolyvagin systems 
for $\Zp(1) \otimes \omega_F$.
\end{defn}

\begin{rem}
Let $\N_p^+ := \{n \in \N_p : \text{all $\ell \mid n$ split in $F/\Q$}\}$.
In \cite{kolysys}, a Kolyvagin system was defined to be a collection of classes 
$\{\kappa_n \in (\hat{F}^\times_\ell)^-\otimes (\otimes_{\ell \mid n} \Gamma_\ell) : n \in \N_p^+\}$, 
and $\phifs_\ell$ took values in $(\hat{F}^\times_{\ell,\tr})^- \otimes \Gamma_\ell$.  
We use Proposition \ref{crit}(iv) to replace $\otimes_{\ell \mid n_+} \Gamma_\ell$ by 
$\In$ and \eqref{caug} to replace $\Gamma_\ell$ by $I_\ell/I_\ell^2$, 
which will be more convenient for our purposes here.  Also, 
a Kolyvagin system $\{\kappa_n : n \in \N_p^+\}$ as in \cite{kolysys} extends uniquely to 
$\{\kappa_n : n \in \N_p\}$ simply by setting $\kappa_n := \kappa_{n_+}$ for $n \in \N_p - \N_p^+$.

\end{rem}

The following theorem is the key to our proof of Theorem \ref{mainthm}.

\begin{thm}
\label{rankone}
Suppose $\ksys, \ksys' \in \KS(F)$.  
If $\kappa_1 = \kappa'_1$, then $\kappa_n = \kappa'_n$ for every $n \in \N_p$.
\end{thm}

\begin{proof}
We follow \S6.1 of \cite{kolysys}, with $R = \Zp$, $\rho = \omega_F$, 
$T = \Zp(1) \otimes \omega_F$, 
and with the Selmer structure denoted $\cF$ in \cite{kolysys}.  
By Lemma 6.1.5 and Proposition 6.1.6 of \cite{kolysys}, the hypotheses needed to 
apply the results of \S5.2 of \cite{kolysys} all hold, and the core rank of $T$ is $1$.  

By Theorem 5.2.10(ii) of \cite{kolysys}, $\KS(F)$ is a free $\Zp$-module 
of rank one.  Therefore (switching $\ksys$ and $\ksys'$ if necessary) 
there is an $a \in \Zp$ such that $\ksys' = a\ksys$, i.e., $\kappa'_n = a\kappa_n$ 
for every $n \in \N_p$.  If $\ksys$ is identically zero, then so is $\ksys'$ and 
we are done.  If $\ksys$ is not identically zero, then (since the ideal class group 
of $F$ is finite) Theorem 5.2.12(v) of \cite{kolysys} shows that $\kappa_1 \ne 0$.
Since $\kappa'_1 = \kappa_1$ in the torsion-free  
$\Zp$-module $(\hat{F}^\times)^-$ (in fact property (i) above shows that 
$\kappa_1 \in (\O_F^\times \otimes \Zp)^-$), we must have $a = 1$.
\end{proof}

\section{Pre-Kolyvagin systems}
\label{prekssect}

Keep the fixed odd prime $p$.  
The right-hand and left-hand sides of Conjecture \ref{dconj} are ``almost''
Kolyvagin systems.  If they were Kolyvagin systems, then since they agree when 
$n=1$ (Proposition \ref{n=1}), they would agree for all $n$ by Theorem \ref{rankone}, 
and Theorem \ref{mainthm} would be proved.  

In this section we define what we call ``pre-Kolyvagin systems'', 
and show that a pre-Kolyvagin system can be transformed into a Kolyvagin system.  
Using Theorem \ref{rankone}, we deduce (Corollary \ref{prekskscor} below) that if 
two pre-Kolyvagin systems agree when $n=1$, then they agree for every $n$.  
In \S\ref{cupks} and \S\ref{rpks}, respectively, we will show that the left- and right-hand sides 
of Conjecture \ref{dconj} are pre-Kolyvagin systems.  Then Theorem \ref{mainthm} will 
follow from Corollary \ref{prekskscor} and Proposition \ref{n=1}.

If $x \in (\hat{F}^\times)^- \otimes I_n^r/I_n^{r+1}$, let 
$x_\ell$ denote the image of 
$x$ in $(\hat{F}^\times_{\ell})^- \otimes I_n^r/I_n^{r+1}$, 
and if $\ell \in \N_p$ splits in $F/\Q$, let 
$x_{\ell,\f} \in (\hat{F}^\times_{\ell,\f})^- \otimes I_n^r/I_n^{r+1}$ and 
$x_{\ell,\tr}\in (\hat{F}^\times_{\ell,\tr})^- \otimes I_n^r/I_n^{r+1}$ 
denote the projections of $x_\ell$ induced by the splitting \eqref{ftrs}.  
Let $\projI{n}{x} \in (\hat{F}^\times)^- \otimes \In$ denote the projection 
of $x$ induced by the splitting of Proposition \ref{crit}(i), 
and similarly for $\projI{n}{x_\ell}$ and $\projI{n}{x_{\ell,\f}}$.

\begin{defn}
\label{Mdef}
If $n \in \N$ and $d = \prod_{i=1}^t \ell_i$ divides $n_+$, let $M_{n,d} = (m_{ij})$ 
be the $t \times t$ matrix with entries in $I_n/I_n^2$
$$
m_{ij} = 
\begin{cases}
\pi_{n/d}(\Frob_{\ell_i}-1) & \text{if $i = j$},\\
\pi_{\ell_j}(\Frob_{\ell_i}-1) & \text{if $i \ne j$}.
\end{cases}
$$
We let $M_d := M_{d,d}$, where $\pi_1(\Frob_\ell-1)$ is understood to be zero, 
so that all diagonal entries of $M_d$ are zero.  Define
$$
\D_{n,d} := \det(M_{n,d}) \in I_n^t/I_n^{t+1}, \quad 
   \D_d := \det(M_d) \in \I_d \subset I_n^t/I_n^{t+1}.
$$
By convention we let $\D_{n,1} = \D_1 = 1$.
Note that $\D_{n,d}$ and $\D_d$ are 
independent of the ordering of the prime factors of $d$.
\end{defn}

\begin{defn}
\label{preksdef}
A {\em pre-Kolyvagin system} $\ksys$ 
(for $\Zp(1) \otimes \omega_F$) is a collection
$$
\{\kappa_n \in (\hat{F}^\times)^- \otimes I_n^r/I_n^{r+1} : n \in \N_p\}
$$
where $r=r(n)$,
satisfying the following properties for every rational prime $\ell$:  
\begin{enumerate}
\item
If $\ell \nmid n$, then 
$(\kappa_{n})_\ell \in (\hat{F}^\times_{\ell,\f})^- \otimes I_n^r/I_n^{r+1}$.
\item
If $\ell \mid n_+$, then 
$
(1 \otimes \pi_{n/\ell})\kappa_n = \kappa_{n/\ell}\,\pi_{n/\ell}(1-\Frob_\ell).
$
\item
If $\ell \mid n_+$, then
$\projI{n}{(\kappa_{n})_{\ell,\tr}} = (\phifs_\ell \otimes 1)(\projI{n/\ell}{(\kappa_{n/\ell})_\ell})$.
\item
If $\ell \mid n_+$, then $\sum_{d \mid n_+} \projI{n/d}{(\kappa_{n/d})_{\ell,\f}}\,\D_d = 0$.
\item
If $\ell \mid n/n_+$, then $\projI{n}{\kappa_{n}} = \projI{n/\ell}{\kappa_{n/\ell}}$.

\end{enumerate}
Let $\preKS(F)$ denote the $\Zp$-module of pre-Kolyvagin systems 
for $\Zp(1) \otimes \omega_F$.

\end{defn}

\begin{defn}
\label{pktok}
If $\ksys = \{\kappa_n : n \in \N_p\}$ is a pre-Kolyvagin system, 
define $\tilde{\ksys} = \{\tilde\kappa_n : n \in \N_p\}$ by
$$
\tilde\kappa_n := \sum_{d \mid n_+} \kappa_{n/d}\,\D_{n,d}.
$$
\end{defn}

\begin{lem}
\label{detMlem}
Suppose $n \in \N_p$ and $d \mid n$.
\begin{enumerate}
\item
If $\ell \mid d$ then $\pi_{n/\ell}(\D_{n,d}) = \pi_{n/d}(\Frob_\ell-1)\D_{n/\ell,d/\ell}$.
\item
If $\ell \nmid d$ then $\pi_{n/\ell}(\D_{n,d}) = \D_{n/\ell,d}$.
\item
$\pi_{d}(\D_{n,d}) = \D_d \in \I_d$.
\end{enumerate}
\end{lem}

\begin{proof}
Suppose $\ell \mid d$.  The column of $\pi_{n/\ell}(M_{n,d})$ corresponding to 
$\ell$ consists of all zeros except for $\pi_{n/d}(\Frob_\ell-1)$ on the diagonal.  The 
first assertion follows from this, and (ii) and (iii) follow directly from the definition.
\end{proof}

\begin{prop}
\label{preksks}
The map $\ksys \mapsto \tilde\ksys$ of Definition \ref{pktok} 
is a $\Zp$-module isomorphism $\preKS(F) \isom \KS(F)$ between free $\Zp$-modules of rank one.
\end{prop}

\begin{proof}
The $\Zp$-linearity is clear.  The injectivity is clear as well, since it follows easily 
by induction that if $\tilde\kappa_n = 0$ for all $n$, then $\kappa_n = 0$ for all $n$.  

We next show that if $\ksys$ is a pre-Kolyvagin system, then $\tilde\ksys$ is 
a Kolyvagin system.  In other words, we need to show for every $n \in \N_p$ that 
\begin{enumerate}
\renewcommand{\theenumi}{(\alph{enumi})}
\item
$\tilde\kappa_n \in (\hat{F}^\times)^- \otimes \In$, 
\item
if $\ell \nmid n$ then $(\tilde\kappa_n)_{\ell} \in (\hat{F}^\times_{\ell,\f})^- \otimes \In$,
\item
if $\ell \mid n_+$ then $(\tilde\kappa_n)_{\ell,\tr} = (\phifs_\ell \otimes 1)((\kappa_{n/\ell})_\ell)$,
\item
if $\ell \mid n_+$ then $(\tilde\kappa_n)_{\ell,\f} = 0$,
\item
if $\ell \mid n/n_+$ then $\tilde\kappa_n = \tilde\kappa_{n/\ell}$.
\end{enumerate}

Fix $n \in \N_p$, and suppose that $\ell \mid n_+$.  Then
\begin{align*}
(1 \otimes \pi_{n/\ell})(\tilde\kappa_n) 
   &= \sum_{d \mid n_+, \ell\nmid d} (1 \otimes \pi_{n/\ell})(\kappa_{n/d} \D_{n,d}) 
      + \sum_{d \mid n_+,\ell\mid d} (1 \otimes \pi_{n/\ell})(\kappa_{n/d} \D_{n,d}) \\
   &= \sum_{d \mid (n_+/\ell)} \kappa_{n/(d\ell)} \,\pi_{n/\ell}(\D_{n,d\ell})
      + (1 \otimes \pi_{n/(d\ell)})(\kappa_{n/d}) \pi_{n/\ell}(\D_{n,d}).
\end{align*}
Fix a divisor $d$ of $n_+/\ell$.  By Lemma \ref{detMlem}(i), 
$$
\kappa_{n/(d\ell)} \,\pi_{n/\ell}(\D_{n,d\ell}) = 
   \kappa_{n/(d\ell)}\,\pi_{n/(d\ell)}(\Frob_\ell-1) \D_{n/\ell,d}.
$$
Also, 
$(1 \otimes \pi_{n/(d\ell)})(\kappa_{n/d}) = \kappa_{n/(d\ell)}\,\pi_{n/(d\ell)}(1-\Frob_\ell)$
by Definition \ref{preksdef}(ii), so by Lemma \ref{detMlem}(ii)
$$
(1 \otimes \,\pi_{n/(d\ell)})(\kappa_{n/d}) \pi_{n/\ell}(\D_{n,d}) 
   = \kappa_{n/(d\ell)}\,\pi_{n/(d\ell)}(1-\Frob_\ell)\D_{n/\ell,d}.
$$
Thus $(1 \otimes \pi_{n/\ell})(\tilde\kappa_n) = 0$ 
for every $\ell$ dividing $n$.  Since $(\hat{F}^\times)^-$ is a free $\Zp$-module, 
it follows from Proposition \ref{crit}(iii) that 
$\tilde\kappa_n \in (\hat{F}^\times)^- \otimes \In$.  This is property (a) above.

By (a), and using that $\pi_d(\D_{n,d}) \in \I_d$, we have 
$$
\tilde\kappa_n = \projI{n}{\tilde\kappa_n} 
   = \sum_{d \mid n_+}\projI{n/d}{\kappa_{n/d}}\,\pi_{d}(\D_{n,d}).
$$
If $\ell \nmid n$, then property (i) of Definition \ref{preksdef} of a pre-Kolyvagin system 
shows that $\projI{n/d}{(\kappa_{n,d})_\ell} \in (\hat{F}^\times_{\ell,\f})^- \otimes \I_{n/d}$ 
for every $d$, so
$
(\tilde\kappa_n)_\ell \in (\hat{F}^\times_{\ell,\f})^- \otimes \In.
$
This is (b).

Now suppose $\ell \mid n_+$.  
For (c), using property (i) of Definition \ref{preksdef} we have
$$
(\tilde\kappa_n)_{\ell,\tr} = \sum_{d \mid n_+} (\kappa_{n/d})_{\ell,\tr}\D_{n,d}
   = \sum_{d \mid (n_+/\ell)} (\kappa_{n/d})_{\ell,\tr}\D_{n,d}. 
$$
Projecting into $\In$, and using (a), (ii) of Definition \ref{preksdef}, 
and Lemma \ref{detMlem}(ii), we have
\begin{align*}
(\tilde\kappa_n)_{\ell,\tr} &= \projI{n}{(\tilde\kappa_n)_{\ell,\tr}} 
   =  \sum_{d \mid (n_+/\ell)} \projI{n}{(\kappa_{n/d})_{\ell,\tr}\D_{n,d}} \\
   &= \sum_{d \mid (n_+/\ell)} \projI{n}{(\phifs_\ell \otimes 1)((\kappa_{n/(d\ell)})_\ell)
      \,\pi_{n/\ell}(\D_{n,d})} \\
   &= \sum_{d \mid (n_+/\ell)} 
      \projI{n}{(\phifs_\ell \otimes 1)((\kappa_{n/(d\ell)})_\ell)\D_{n/\ell,d}} \\
   &= \projI{n}{(\phifs_\ell \otimes 1)(\tilde\kappa_{n/\ell})} 
      = (\phifs_\ell \otimes 1)(\projI{n/\ell}{\tilde\kappa_{n/\ell}})
      = (\phifs_\ell \otimes 1)(\tilde\kappa_{n/\ell}).
\end{align*}
This is (c).  For (d), using (a), Lemma \ref{detMlem}(iii), and 
(iv) of Definition \ref{preksdef} we have
$$
(\tilde\kappa_n)_{\ell,\f} = \projI{n}{(\tilde\kappa_n)_{\ell,\f}} 
   = \sum_{d \mid n_+} \projI{n/d}{(\kappa_{n/d})_{\ell,\f}}\projI{d}{\pi_{d}(\D_{n,d})} 
   = \sum_{d \mid n_+} \projI{n/d}{(\kappa_{n/d})_{\ell,\f}}\D_{d} = 0.
$$

Finally, suppose that $\ell \mid n/n_+$.  Using Definition \ref{preksdef}(v) 
and property (a) above,
$$
\tilde\kappa_n = \projI{n}{\tilde\kappa_n} 
   = \sum_{d \mid n_+} \projI{n/d}{(\kappa_{n/d})}\D_{d} 
   = \sum_{d \mid (n/\ell)_+} \projI{n/(d\ell)}{(\kappa_{n/(d\ell)})}\D_{d}
= \projI{n/\ell}{\tilde\kappa_{n/\ell}} = \tilde\kappa_{n/\ell}.
$$
This completes the proof that $\tilde\ksys$ is a Kolyvagin system.

Since $\KS(F)$ is a free $\Zp$-module of rank one \cite[Theorem 5.2.10(ii)]{kolysys}, 
to complete the proof it remains only to show that the map 
$\preKS(F) \to \KS(F)$ is surjective.  
If $\tilde\ksys \in \KS(F)$, then (since $\D_{n,1} = 1$) we can define  
inductively a collection 
$\ksys := \{\kappa_n \in (\hat{F}^\times)^- \otimes I_n^r/I_n^{r+1} : n \in \N_p\}$ such that 
$\sum_{d \mid n_+} \kappa_{n/d}\,\D_{n,d} = \tilde\kappa_n$ for every $n$.  
It is straightforward to check that $\ksys$ is a pre-Kolyvagin system; since we 
will not make use of this, we omit the proof.  By Definition \ref{pktok} the image of 
$\ksys$ in $\KS(F)$ is $\tilde\ksys$.
\end{proof}

\begin{cor}
\label{prekskscor}
Suppose $\ksys, \ksys' \in \preKS(F)$.  
If $\kappa_1 = \kappa'_1$, then $\kappa_n = \kappa'_n$ for every $n \in \N_p$.
\end{cor}

\begin{proof}
Let $\tilde\ksys$ and $\tilde\ksys'$ be the images of $\ksys$ and $\ksys'$, respectively, 
under the map of Definition \ref{pktok}.  Then $\tilde\ksys$ and $\tilde\ksys'$ are 
Kolyvagin systems, and $\tilde\kappa_1 = \kappa_1 = \kappa'_1 = \tilde\kappa'_1$.  
Therefore $\tilde\ksys = \tilde\ksys'$ by Theorem \ref{rankone}, so by the injectivity 
assertion of Proposition \ref{preksks} we have $\ksys = \ksys'$, i.e., 
$\kappa_n = \kappa'_n$ for every $n \in \N_p$.
\end{proof}

We will use the following definition and lemma to replace property (iv) in the 
definition of a pre-Kolyvagin system by an equivalent property that will be 
easier to verify.  See Remark \ref{5.9} below.

\begin{defn}
\label{symmdef}
If $n \in \N$, let $\symm(n)$ denote the set of permutations of the primes 
dividing $n_+$, and let $\symm_1(n) \subset \symm(n)$ be the subset
$$
\symm_1(n) := \{\sigma \in \symm(n) : 
    \text{the primes not fixed by $\sigma$ form a single $\sigma$-orbit}\}.
$$
If $\sigma \in \symm(n)$ let $d_\sigma := \prod_{\ell \mid n_+, \sigma(\ell) \ne \ell} \ell$, 
the product of the primes not fixed by $\sigma$, and define
$$
\Pi(\sigma) := \prod_{q \mid d_\sigma}\pi_q(\Frob_{\sigma(q)}-1).
$$
\end{defn}

\begin{lem}
\label{5.8}
Suppose that $A$ is an abelian group, $\ell$ is a prime that splits in $F/\Q$, 
and $x_n \in A \otimes \In$ for every $n \in \N_p$.  
Then the following are equivalent:
\begin{enumerate}
\item
For every $n$ divisible by $\ell$, $\sum_{d \mid n_+} x_{n/d}\,\D_d = 0$.
\item
For every $n$ divisible by $\ell$, 
$\displaystyle x_n = 
   -\sum_{\pile{\sigma\in\symm_1(n)}{\sigma(\ell) \ne \ell}} 
        \sign(\sigma) x_{n/d_\sigma} \Pi(\sigma)$.
\end{enumerate}
\end{lem}

\begin{proof}
We show first that (ii) implies (i) (which is the implication we use later in this paper).
Let $\symm'(d) \subset \symm(d)$ denote the derangements, i.e., the permutations 
with no fixed points.
Then we can evaluate the determinant $\D_{d} = \det(M_d)$ as follows.  
Let $m_{q,q'}$ be the $(q,q')$-entry in $M_d$.  Then
\begin{equation}
\label{detp}
\D_d = \sum_{\sigma \in \symm(d)} \sign(\sigma) \prod_{q \mid d} m_{q,\sigma(q)}
   = \sum_{\sigma \in \symm'(d)} \sign(\sigma)\Pi(\sigma),
\end{equation}
where the second equality holds since the diagonal entries of $M_d$ vanish.

Fix an $n$ divisible by $\ell$, and let
$$
S_1 = \sum_{d \mid n_+, \ell \nmid d} x_{n/d}\,\D_d, \quad S_2 
   = \sum_{d \mid n_+, \ell \mid d} x_{n/d}\,\D_d.
$$
Using property (ii) we have
\begin{equation}
\label{cancel}
S_1 = -\sum_{\pile{d \mid n_+}{\ell\nmid d}}
      \sum_{\pile{\sigma\in\symm_1(n/d)}{\sigma(\ell) \ne \ell}} 
      \sign(\sigma) \projI{n/(dd_\sigma)}{(x_{n/(dd_\sigma)})_{\ell}} 
      \Pi(\sigma)\D_d. 
\end{equation}
Fix a divisor $\delta$ of $n_+$ that is divisible by $\ell$.  We will show that the coefficient 
of $x_{n/\delta}$ in $S_1$ in \eqref{cancel} is $-\D_\delta$, which exactly cancels 
the coefficient of $x_{n/\delta}$ in $S_2$.
Using \eqref{detp}, the coefficient of $x_{n/\delta}$ in $S_1$ in \eqref{cancel} is
$$
-\sum_{d \mid (\delta/\ell)}\sum_{\pile{\sigma\in\symm_1(n/d)}{d_\sigma=\delta/d}}
      \biggl(\sign(\sigma)\Pi(\sigma)\sum_{\eta \in \symm'(d)}\sign(\eta)\Pi(\eta)\biggr) 
   = -\sum_{d \mid (\delta/\ell)}\sum_{\pile{\sigma\in\symm_1(n/d)}{d_\sigma=\delta/d}}
      \sum_{\eta \in \symm'(d)}\sign(\sigma\eta)\Pi(\sigma\eta).
$$
For every $\rho \in \symm'(\delta)$ there is a unique triple $(d,\sigma,\eta)$ such that
$$
\text{$d \mid \delta/\ell$, \;$\sigma \in \symm_1(n/d)$, \;$d_\sigma = \delta/d$, 
   \;$\eta \in\symm'(d)$, \;and \;$\rho = \sigma\eta$}.
$$
To see this, 
simply write $\rho$ as a product of disjoint cycles, let $\sigma$ be the cycle containing 
$\ell$, and let $d = \delta/d_\sigma$ and $\eta = \sigma^{-1}\rho$.  Thus 
the coefficient of $x_{n/\delta}$ in $S_1$ in \eqref{cancel} is (using \eqref{detp} again)
$$
-\sum_{\rho \in \symm'(\delta)} \sign(\rho)\Pi(\rho) = -\D_\delta. 
$$
Therefore $\sum_{d \mid n_+} x_{n/d}\,\D_d = S_1 + S_2 = 0$, so (i) holds.

Although we will not need it, here is a simple argument to show that (i) implies (ii).  
Suppose that $X := \{x_n \in A \otimes \In : n \in \N_p\}$ 
satisfies (i).  
If $\ell \mid n$, then (since $\D_1 = 1$) we can use (i) recursively to express $x_n$ 
as a linear combination of $x_d$ with $\ell \nmid d$.  Thus $X$ is uniquely 
determined by the subset $X' := \{x_n \in A \otimes \In : n \in \N_p, \ell \nmid n\}$.  
Clearly $X'$ determines a unique collection $Y := \{y_n \in A \otimes \In : n \in \N_p\}$ 
satisfying (ii), with $y_n = x_n$ if $\ell \nmid n$.  We showed above 
that (ii) implies (i), so $Y$ satisfies (i).  Since (i) and 
$X'$ uniquely determine both $X$ and $Y$, we must have $X = Y$, and so $X$ satisfies (ii).
\end{proof}

\begin{rem}
\label{5.9}
We will apply Lemma \ref{5.8} as follows.  Let $A := (\hat{F}^\times_{\ell,\f})^-$, and let 
$x_n := \projI{n}{(\kappa_n)_{\ell,\f}}$.  Then Lemma \ref{5.8} says that we 
can replace property (iv) in Definition \ref{preksdef} of a pre-Kolyvagin system 
by the equivalent statement:
\begin{enumerate}
\item[(iv)$'$]
if $\ell \mid n_+$, then
$
\displaystyle
\projI{n}{(\kappa_{n})_{\ell,\f}} = 
   -\sum_{\pile{\sigma\in\symm_1(n)}{\sigma(\ell) \ne \ell}} 
        \sign(\sigma) \projI{n/d_\sigma}{(\kappa_{n/d_\sigma})_{\ell}} 
        \Pi(\sigma).
$
\end{enumerate}
\end{rem}

\section{The cyclotomic unit pre-Kolyvagin system}
\label{cupks}

Fix an odd prime $p$.
If $n \in \N$, let $s(n)$ be the number of prime factors of $n/n_+$.  
In this section we will show 
that the collection $\{2^{-s(n)}\tilde\theta'_n : n \in \N_p\}$ is a 
pre-Kolyvagin system.  Recall that 
$$
\N_p^+ := \{n \in \N_p : \text{all $\ell \mid n$ split in $F/\Q$}\}.
$$

\begin{prop}[(Darmon)]
\label{9.4}
If $n \in \N_p$ then
$$
\sum_{d \mid n_+} \tilde\theta'_{n/d} \prod_{\ell \mid d}\pi_{n/d}(\Frob_\ell-1) 
   = 2^{s(n)}\beta_{n_+} \quad \text{in $(\hat{F}^\times)^- \otimes \In$}
$$
where for $n \in \N_p^+$, $\beta_{n} \in (\hat{F}^\times)^- \otimes \In$ is the Kolyvagin 
derivative class denoted $\kappa(n)$ in \cite[\S6]{darmon}, or $\kappa_{n}$ in 
\cite[Appendix]{babykolysys}.  
\end{prop}

\begin{proof}
This is Proposition 9.4 of \cite{darmon}.\footnote{There is a typo in 
\cite[Proposition 9.4]{darmon}.  The last two $T$'s should be $TQ$, as in 
\cite[Lemma 8.1]{darmon}.}
(Note that $\kappa(n)$ in \cite[\S6]{darmon} and $\kappa_{n}$ in 
\cite[Appendix]{babykolysys} are defined to lie in 
$(\hat{F}^\times)^- \otimes (\Z/\gcd(\ell-1: \ell | n)\Z)$, after fixing generators of every 
$\Gamma_\ell$.  Without fixing such choices, the elements defined in \cite{darmon} 
and \cite{babykolysys} live naturally in $(\hat{F}^\times)^- \otimes \In$.)
\end{proof}

\begin{thm}
\label{thetapreks}
The collection $\{2^{-s(n)}\tilde\theta'_n : n \in \N_p\}$ is a 
pre-Kolyvagin system.
\end{thm}

\begin{proof}
We need to check the five properties of Definition \ref{preksdef}.  
For $n \in \N_p^+$, let $\beta_{n}$ be as in Proposition \ref{9.4}.

Since $\beta_{n_+} \in (\hat{F}^\times)^- \otimes \In$ for every $n$, it follows 
easily by induction from Proposition \ref{9.4} that 
$\tilde\theta'_n \in (\hat{F}^\times)^- \otimes I_n^r/I_n^{r+1}$, 
where $r$ is the number of prime factors of $n_+$.
This is property (i) of Definition \ref{preksdef}.

Suppose $\ell \mid n_+$.  A standard property of cyclotomic units shows that
$$
\bN_{F(\bmu_n)/F(\bmu_{n/\ell})}\alpha_n = \alpha_{n/\ell}/\alpha_{n/\ell}^{\Frob_\ell^{-1}}.
$$
It follows from the definition of $\theta'_n$ that
\begin{multline*}
(1 \otimes \pi_{n/\ell})(\theta'_n) 
   = \sum_{\gamma \in \Gamma_n} \gamma(\alpha_n) \otimes \pi_{n/\ell}(\gamma)
   = \sum_{\gamma \in \Gamma_{n/\ell}} \gamma(\bN_{F(\bmu_n)/F(\bmu_{n/\ell})}\alpha_n) \otimes \gamma\\
   = \sum_{\gamma \in \Gamma_{n/\ell}} 
      \gamma\bigl(\alpha_{n/\ell}/\alpha_{n/\ell}^{\Frob_\ell^{-1}}\bigr) \otimes \gamma
   = \sum_{\gamma \in \Gamma_{n/\ell}} \gamma(\alpha_{n/\ell}) \otimes \gamma\,\pi_{n/\ell}(1-\Frob_\ell)
   = \theta'_{n/\ell}\,\pi_{n/\ell}(1-\Frob_\ell).
\end{multline*}
Since $\ell \mid n_+$ we have $s(n) = s(n/\ell)$, so 
this verifies property (ii) of Definition \ref{preksdef}.

Projecting each of the summands in Proposition \ref{9.4} into $(\hat{F}^\times)^- \otimes \In$, 
one sees that all terms with $d > 1$ vanish, yielding 
$$
\projI{n}{2^{-s(n)}\tilde\theta'_{n}} = \projI{n}{\beta_{n_+}} = \beta_{n_+}.
$$
Properties (iii), (iv), and (v) of Definition \ref{preksdef} follow from the 
corresponding properties of the $\beta_{n_+}$.
See \cite[Proposition A.2]{kolysys} or \cite[Theorem 4.5.4]{eulersys} 
for (iii), and \cite[Theorem A.4]{kolysys}
or \cite[Proposition A.2]{babykolysys} for property (iv)$'$ of Remark \ref{5.9}.  
Property (v) is immediate, since $\beta_{n_+}$ depends only on $n_+$.
\end{proof}

\section{The regulator pre-Kolyvagin system}
\label{rpks}

In this section we study relations among the regulator 
elements $R_n$, to show that the collection 
$\{h_nR_n : n \in \N_p\}$ is a pre-Kolyvagin system.

\begin{lem}
\label{cnr}
Suppose $n \in \N$, $\ell \mid n_+$, and   
$\{\lambda_0-\lambda_0^\tau,\ldots,\lambda_r-\lambda_r^\tau\}$ 
is a standard basis of $X_n^-$ with $\lambda_r\lambda_r^\tau = \ell$.
Then $\{\lambda_0-\lambda_0^\tau,\ldots,\lambda_{r-1}-\lambda_{r-1}^\tau\}$ 
is a standard basis of $X_{n/\ell}^-$, and we can choose an oriented basis 
$\{\epsilon_0,\ldots,\epsilon_r\}$ of $(1-\tau)\E_n$ such that 
$\{\epsilon_0,\ldots,\epsilon_{r-1}\}$ is an oriented basis of $(1-\tau)\E_{n/\ell}$.

With any such bases, $\ord_{\lambda_r}(\epsilon_r) = -h_{n/\ell}/h_n$ and 
$$
\grx{\lambda_r}{n/\ell}{\epsilon_r} = \frac{h_{n/\ell}}{h_n}\,\pi_{n/\ell}(1-\Frob_\ell) \in I_{n/\ell}/I_{n/\ell}^2.
$$
\end{lem}

\begin{proof}
Everything except the final sentence is clear.  
Comparing the determinants of the logarithmic embeddings
$$
(1-\tau)\E_{n/\ell} \map{\xi_{n/\ell}} X_{n/\ell}^-, \qquad (1-\tau)\E_n \map{\xi_n} X_n^-
$$
with respect to our given bases, we see that 
$$
\det(\xi_n) = \log|\epsilon_r|_{\lambda_r}\det(\xi_{n/\ell})
$$
because $\log|\epsilon_i|_{\lambda_r} = 0$ for $0 \le i < r$.  Since our bases 
are oriented, both determinants are positive.  Hence 
$$
|\epsilon_r|_{\lambda_r} = \ell^{-\ord_{\lambda_r}(\epsilon_r)} > 1
$$
so $\ord_{\lambda_r}(\epsilon_r) < 0$.

The exact sequence
$$
(1-\tau)\E_n \map{\ord_{\lambda_r}} \Z \map{\cdot\lambda_r} \Pic(\O_F[\ell/n]) 
   \too \Pic(\O_F[1/n]) \too 0
$$
shows that 
$$
[\Z : \ord_{\lambda_r}(\epsilon_r)\Z] = h_{n/\ell}/h_n,
$$
so $\ord_{\lambda_r}(\epsilon_r) = - h_{n/\ell}/h_n$ as claimed.
Since $F(\bmu_{n/\ell})/F$ is unramified at $\lambda_r$, 
$$
\grx{\lambda_r}{n/\ell}{\epsilon_r}
   = (\Frob_{\ell}^{\ord_{\lambda_r}(\epsilon_r)}) - 1 
   = \ord_{\lambda_r}(\epsilon_r) (\Frob_{\ell} - 1)
   = -h_{n/\ell}/h_n (\Frob_{\ell} - 1)
$$
in $I_{n/\ell}/I_{n/\ell}^2$.
\end{proof}

\begin{prop}
\label{regreg1lem}
Suppose $n \in \N$, $\ell \mid n_+$, and $r = r(n)$.  Then
$$
(1 \otimes \pi_{n/\ell})(h_n R_n) = h_{n/\ell} R_{n/\ell}\,\pi_{n/\ell}(1-\Frob_{\ell}) 
   \in F^\times \otimes I_n^r/I_n^{r+1}.
$$
\end{prop}

\begin{proof}
To compute $R_n$, fix bases for $X_n^-$ and $\E_n^-$ as in Lemma \ref{cnr}.
By definition
$$
R_n := \left|
\begin{array}{ccccccc}
\epsilon_0 & \epsilon_1 & \cdots & \epsilon_r \\
\grx{\lambda_1}{n}{\epsilon_0} & \grx{\lambda_1}{n}{\epsilon_1} & \cdots & \grx{\lambda_1}{n}{\epsilon_r} \\
\vdots & \vdots && \vdots\\
\grx{\lambda_r}{n}{\epsilon_0} & \grx{\lambda_r}{n}{\epsilon_1} & \cdots & \grx{\lambda_r}{n}{\epsilon_r}
\end{array}
\right|,
$$
and then $(1 \otimes \pi_{n/\ell})(R_n)$ is the determinant of the matrix 
obtained by applying $\pi_{n/\ell}$ to rows $2$ through $r+1$ of this matrix.
For $i < r$, $\epsilon_i$ is a unit at $\lambda_r$, so the local Artin symbol 
$[\epsilon_i,F(\bmu_n)_{\lambda_r}/F_{\lambda_r}]$ lies in the inertia group $\Gamma_\ell$.
Hence $\pi_{n/\ell}(\grx{\lambda_r}{n}{\epsilon_i}) = \grx{\lambda_r}{n/\ell}{\epsilon_i} = 0$ 
for $i < r$, and so
$$
(1 \otimes\pi_{n/\ell})(R_n) = \left|
\begin{array}{ccccccc}
\epsilon_0 & \cdots & \epsilon_{r-1} & \epsilon_r \\
\grx{\lambda_1}{n/\ell}{\epsilon_0} & \cdots & \grx{\lambda_1}{n/\ell}{\epsilon_{r-1}} 
   &\grx{\lambda_1}{n/\ell}{\epsilon_r} \\
\vdots && \vdots & \vdots\\
\grx{\lambda_{r-1}}{n/\ell}{\epsilon_0} & \cdots & \grx{\lambda_{r-1}}{n/\ell}{\epsilon_{r-1}} 
   &\grx{\lambda_{r-1}}{n/\ell}{\epsilon_r} \\
0 & \cdots & 0 & \grx{\lambda_{r}}{n/\ell}{\epsilon_r}
\end{array}
\right|.
$$
The upper left $r \times r$ determinant is the one used to define $R_{n/\ell}$, so 
$$
(1 \otimes\pi_{n/\ell})(R_n) = R_{n/\ell} \grx{\lambda_r}{n/\ell}{\epsilon_r}
   = \frac{h_{n/\ell}}{h_n } R_{n/\ell} \, \pi_{n/\ell}(1-\Frob_{\ell})
$$
by Lemma \ref{cnr}.
\end{proof}

Fix an odd prime $p$ as in \S\S\ref{kssect} and \ref{prekssect}, 
and keep the rest of the notation of those sections as well.

\begin{lem}
\label{finlem}
If $n \in \N_p$, $\ell$ is a prime not dividing $n$, and $r = r(n)$, then 
$$
(R_n)_{\ell} \in (\hat{F}_{\ell,\f}^\times)^- \otimes I_n^r/I_n^{r+1}.
$$
\end{lem}

\begin{proof}
Since $\ell \nmid n$, if $\epsilon \in \E_n^-$ then 
$\epsilon_\ell \in (\hat{\O}_\ell^\times)^- 
   = (\hat{F}_{\ell,\f}^\times)^- \subset (\hat{F}_{\ell}^\times)^-$.
Now the lemma is clear, since $R_n \in \E_n^- \otimes I_n^r/I_n^{r+1}$.
\end{proof}

\begin{prop}
\label{fsprop}
Suppose $n \in \N_p$ and $\ell \mid n_+$.
Then 
$$
\projI{n}{h_n(R_n)_{\ell,\tr}} = (\phifs_\ell \otimes 1)(\projI{n/\ell}{h_{n/\ell}(R_{n/\ell})_\ell}).
$$
\end{prop}

\begin{proof}
Note that 
$(\phifs_\ell \otimes 1)(\projI{n/\ell}{h_{n/\ell}(R_{n/\ell})_\ell}) 
   \in (\hat{F}_{\ell,\tr}^\times)^- \otimes \In$ 
is well-defined, since Lemma \ref{finlem} shows that 
$(R_{n/\ell})_{\ell} \in (\hat{F}_{\ell,\f}^\times)^- \otimes I_{n/\ell}^{r-1}/I_{n/\ell}^{r}$. 

As in the proof of Proposition \ref{regreg1lem}, fix a basis
$\{\lambda_0-\lambda_0^\tau,\ldots,\lambda_r-\lambda_r^\tau\}$ of $X_n^-$ 
with $\ell = \lambda_r\lambda_r^\tau$, and 
an oriented basis $\{\epsilon_0,\ldots,\epsilon_r\}$ of $(1-\tau)\E_n$ as in Lemma \ref{cnr}.  Then
$$
(R_n)_{\ell,\tr} = \left|
\begin{array}{ccccccc}
(\epsilon_0)_{\ell,\tr} & \cdots & (\epsilon_{r-1})_{\ell,\tr} & (\epsilon_r)_{\ell,\tr} \\
\grx{\lambda_1}{n}{\epsilon_0} & \grx{\lambda_1}{n}{\epsilon_1} & \cdots & \grx{\lambda_1}{n}{\epsilon_r} \\
\vdots & \vdots && \vdots\\
\grx{\lambda_r}{n}{\epsilon_0} & \grx{\lambda_r}{n}{\epsilon_1} & \cdots & \grx{\lambda_r}{n}{\epsilon_r}
\end{array}
\right| 
   = \ord_{\lambda_r}(\epsilon_r)
\left|
\begin{array}{ccccccc}
1 & \cdots & 1 & (\ell,\ell^{-1}) \\
\grx{\lambda_1}{n}{\epsilon_0} & \grx{\lambda_1}{n}{\epsilon_1} & \cdots & \grx{\lambda_1}{n}{\epsilon_r} \\
\vdots & \vdots && \vdots\\
\grx{\lambda_r}{n}{\epsilon_0} & \grx{\lambda_r}{n}{\epsilon_1} & \cdots & \grx{\lambda_r}{n}{\epsilon_r}
\end{array}
\right|
$$
since $(\epsilon_r)_{\ell,\tr} = (\ell,\ell^{-1})^{\ord_{\lambda_r}(\epsilon_r)}$, 
and $(\epsilon_i)_{\ell,\tr} = 1$ for $i < r$. 
(Recall that when we evaluate these determinants using \eqref{minors}, the multiplicative notation 
in $(\hat{F}^\times_\ell)_\tr$ changes to additive notation in the tensor product
$(\hat{F}^\times_\ell)_\tr \otimes I_\ell^r/I_\ell^{r+1}$, so $1$'s in the 
top row become $0$'s, and $(\ell,\ell^{-1})^{\ord_{\lambda_r}(\epsilon_r)}$ becomes 
${\ord_{\lambda_r}(\epsilon_r)}\cdot(\ell,\ell^{-1})$.)
We have $\ord_{\lambda_r}(\epsilon_r) = -h_{n/\ell}/h_n$ by Lemma \ref{cnr}.  
For $i < r$ we have $\ord_{\lambda_r}(\epsilon_i) = 0$, so 
$\grx{\lambda_r}{n}{\epsilon_i} = \grx{\lambda_r}{\ell}{\epsilon_i} \in I_\ell/I_\ell^2$ and
$$
\phifs_\ell((\epsilon_i)_{\ell}) = (\ell,\ell^{-1}) \otimes \grx{\lambda_r}{n}{\epsilon_i}
   \in (\hat{F}^\times_\ell)_\tr \otimes I_\ell/I_\ell^2.
$$ 
Thus 
\begin{align*}
(R_n)_{\ell,\tr} &= -\frac{h_{n/\ell}}{h_n}(-1)^{r}(-1)^{r-1}\left|
\begin{array}{ccccccc}
\phifs_\ell((\epsilon_0)_{\ell}) & \cdots & \phifs_\ell((\epsilon_{r-1})_{\ell}) \\
\grx{\lambda_1}{n}{\epsilon_0} & \cdots & \grx{\lambda_1}{n}{\epsilon_{r-1}} \\
\vdots & \vdots & \vdots\\
\grx{\lambda_{r-1}}{n}{\epsilon_0} & \cdots & \grx{\lambda_{r-1}}{n}{\epsilon_{r-1}}
\end{array}
\right|\\[7pt]
   &= \frac{h_{n/\ell}}{h_n} \left|
\begin{array}{ccccccc}
\phifs_\ell((\epsilon_0)_{\ell}) & \cdots & \phifs_\ell((\epsilon_{r-1})_{\ell}) \\
\grx{\lambda_1}{n/\ell}{\epsilon_0}+\grx{\lambda_1}{\ell}{\epsilon_0} & \cdots 
   & \grx{\lambda_1}{n/\ell}{\epsilon_{r-1}}+\grx{\lambda_1}{\ell}{\epsilon_{r-1}} \\
\vdots & \vdots & \vdots\\
\grx{\lambda_{r-1}}{n/\ell}{\epsilon_0}+\grx{\lambda_{r-1}}{\ell}{\epsilon_0} & \cdots 
   & \grx{\lambda_{r-1}}{n/\ell}{\epsilon_{r-1}}+\grx{\lambda_{r-1}}{\ell}{\epsilon_{r-1}} 
\end{array}
\right|
\end{align*}
(the $(-1)^r$ because we moved column $r+1$ to column $1$, and the  $(-1)^{r-1}$ 
because we moved row $r+1$ to row $2$).
When we expand the last determinant (including expanding the sums 
$\grx{\lambda_{j}}{n/\ell}{\epsilon_{i}}+\grx{\lambda_{j}}{\ell}{\epsilon_{i}}$), 
each term that 
includes one of the $\grx{\lambda_j}{\ell}{\epsilon_i}$ lies in $I_\ell^2$ 
(since the top row also contributes one element of $I_\ell$).  Thus all such terms 
project to zero in $\In$, and so 
$$
\projI{n}{(R_n)_{\ell,\tr}} = \frac{h_{n/\ell}}{h_n} ~\projI{n}{\det(A)}
$$
where
$$
A = \left[
\begin{array}{ccccccc}
\phifs_\ell((\epsilon_0)_{\ell}) & \cdots & \phifs_\ell((\epsilon_{r-1})_{\ell}) \\
\grx{\lambda_1}{n/\ell}{\epsilon_0} & \cdots & \grx{\lambda_1}{n/\ell}{\epsilon_{r-1}} \\
\vdots & \vdots & \vdots\\
\grx{\lambda_{r-1}}{n/\ell}{\epsilon_0} & \cdots & \grx{\lambda_{r-1}}{n/\ell}{\epsilon_{r-1}}
\end{array}
\right].
$$
But then $\det(A) = (\phifs_\ell \otimes 1)((R_{n/\ell})_\ell)$, so the proposition follows.
\end{proof}

Suppose $n, n' \in \N$, $n \mid n'$, and $r = r(n)$.  
Define
$$
\bR{n}{n'} := \left|
\begin{array}{ccccccc}
\epsilon_0 & \epsilon_1 & \cdots & \epsilon_r \\
\grx{\lambda_1}{n'}{\epsilon_0} & \grx{\lambda_1}{n'}{\epsilon_1} & \cdots & \grx{\lambda_1}{n'}{\epsilon_r} \\
\vdots & \vdots && \vdots\\
\grx{\lambda_r}{n'}{\epsilon_0} & \grx{\lambda_r}{n'}{\epsilon_1} & \cdots & \grx{\lambda_r}{n'}{\epsilon_r}
\end{array}
\right| \in \E_n^- \otimes I_{n'}^r/I_{n'}^{r+1},
$$
using any standard basis of $X_n^-$ and oriented basis of $(1-\tau)\E_n$.
In particular $\bR{n}{n} = R_n$.

\begin{prop}
\label{indstep}
Suppose $n \in \N$ and $\ell \nmid n$.
\begin{enumerate}
\item
If $\ell$ is inert in $F/\Q$, then $h_{n\ell}\projI{n}{R_{n\ell}} = h_n\projI{n}{R_n}$.
\item
If $\ell$ splits in $F/\Q$ and $v \in I_n$, then
$$
h_{n}\projI{n\ell}{\bR{n}{n\ell}\,v} = \projI{n}{R_n}\pi_\ell(v) 
   - \sum_{\mathrm{primes}~ q \mid n_+}h_{n/q}\projI{n}{\bR{n/q}{n}\,v}\pi_{\ell}(\Frob_q-1)
$$
in $\E_n^- \otimes \I_{n\ell}$.
\end{enumerate}
\end{prop}

\begin{proof}
Let $r$ be the number of prime divisors of $n_+$, so $X_{n}^-$ and 
$(1-\tau)\E_{n}$ are free $\Z$-modules of rank $r+1$.
Choose a standard basis of $X_{n}^-$ and an oriented basis of $(1-\tau)\E_{n}$.  
For $1 \le i \le r = r(n)$, let 
$$
a_i = (\grx{\lambda_i}{n}{\epsilon_0}, \grx{\lambda_i}{n}{\epsilon_1}, \ldots, 
   \grx{\lambda_i}{n}{\epsilon_r}), 
\quad 
b_i = (\grx{\lambda_i}{\ell}{\epsilon_0}, \grx{\lambda_i}{\ell}{\epsilon_1}, \ldots, 
   \grx{\lambda_i}{\ell}{\epsilon_r}).
$$
Then
\begin{equation}
\label{5}
\bR{n}{n\ell} = \left|
\begin{array}{ccccccc}
\epsilon_0 & \cdots & \epsilon_r \\
&a_1 + b_1 \\
& \vdots \\
& a_r+ b_r
\end{array}
\right|
=
\sum_{T \subset \{1,\ldots,r\}} \det(A_T)
\end{equation}
where $A_T$ is the matrix whose top row is $(\epsilon_0, \ldots, \epsilon_r)$ and 
whose $(i+1)$-st row for $1 \le i \le r$ is $b_i$ if $i \in T$ and $a_i$ if $i \notin T$.  
Note that $\det(A_\emptyset) = R_{n}$, and that 
the entries of each $b_i$ are in $I_\ell/I_\ell^2$.

Suppose first that $\ell$ is inert in $F/\Q$, so $(n\ell)_+ = n_+$.  
Then $\projI{n}{\det(A_T)} = 0$ 
if $T$ is nonempty (since $\In$ has no ``$\ell$ component''), so \eqref{5} 
shows that
$$
\projI{n}{\bR{n}{n\ell}} = \projI{n}{\det(A_\emptyset)} = \projI{n}{R_{n}}.
$$
Further, since $\ell$ is inert in $F/\Q$ we have $X_{n\ell}^- = X_{n}^-$, 
$(1-\tau)\E_{n\ell} = (1-\tau)\E_{n}$, and $h_{n\ell} = h_{n}$.  Thus 
$\bR{n}{n\ell} = R_{n\ell}$, and so 
$$
h_{n\ell} \projI{(n\ell)}{R_{n\ell}} = h_n \projI{n}{\bR{n}{n\ell}} 
   = h_{n} \projI{n}{R_{n}}.
$$
This is (i).

Now suppose that $\ell$ splits in $F/\Q$.
Since the entries of each $b_i$ are in $I_\ell$, if $\#(T) \ge 2$ we have 
$\projI{n\ell}{\det(A_T)v} = 0$.  Thus \eqref{5} gives
\begin{equation}
\label{3}
\projI{n\ell}{\bR{n}{n\ell}\,v} = \projI{n\ell}{\det(A_\emptyset)v} 
   + \sum_{i=1}^r \projI{n\ell}{\det(A_{\{i\}})v}.
\end{equation}
By definition of $R_n$, 
\begin{equation}
\label{4}
\projI{n\ell}{\det(A_\emptyset)v} = \projI{n\ell}{R_{n}\,v} 
   = \projI{n}{R_{n}}\pi_\ell(v).
\end{equation} 

To compute $\det(A_{\{i\}})$, 
let $q = \lambda_i\lambda_i^\tau$, and assume that our oriented basis of 
$(1-\tau)\E_{n}$ was chosen so that 
$\{\epsilon_0,\ldots,\epsilon_{r-1}\}$ is an oriented basis of $(1-\tau)\E_{n/q}$ 
with respect to the standard basis of $X_{n/q}$ obtained by removing 
$\lambda_i-\lambda_i^\tau$ from 
$\{\lambda_1-\lambda_1^\tau,\ldots,\lambda_r-\lambda_r^\tau\}$.
For $1 \le j \le r-1$, $\epsilon_j$ is a unit at $\lambda_i$, so 
$\grx{\lambda_i}{\ell}{\epsilon_j} = 0$.  Thus 
\begin{multline*}
\det(A_{\{i\}}) = \left|
\begin{array}{ccccccc}
\epsilon_0 & \cdots & \epsilon_{r-1} & \epsilon_r \\
\grx{\lambda_1}{n}{\epsilon_0} & \cdots & \grx{\lambda_1}{n}{\epsilon_{r-1}} & \grx{\lambda_1}{n}{\epsilon_r} \\
\vdots && \vdots & \vdots\\
0 & \cdots & 0 & \grx{\lambda_i}{\ell}{\epsilon_r} \\
\vdots && \vdots & \vdots \\
\grx{\lambda_r}{n}{\epsilon_0} & \cdots & \grx{\lambda_r}{n}{\epsilon_{r-1}} & \grx{\lambda_r}{n}{\epsilon_r}
\end{array}
\right| 
= (-1)^{r+i} \left|
\begin{array}{ccccccc}
\epsilon_0 & \cdots & \epsilon_{r-1} \\
\grx{\lambda_1}{n}{\epsilon_0} & \cdots & \grx{\lambda_1}{n}{\epsilon_{r-1}} \\
\vdots & \vdots & \vdots\\
\grx{\lambda_{r}}{n}{\epsilon_0} & \cdots & \grx{\lambda_{r}}{n}{\epsilon_{r-1}}
\end{array}
\right| \grx{\lambda_i}{\ell}{\epsilon_r} \\
= (-1)^{r+i} \bR{n/q}{n}\,\grx{\lambda_i}{\ell}{\epsilon_r}
\end{multline*}
(where the second determinant has no $\lambda_i$ row).  Further, an argument 
identical to that of Lemma \ref{cnr} shows that
$$
\grx{\lambda_i}{\ell}{\epsilon_r} = (-1)^{r+i+1}\frac{h_{n/q}}{h_{n}}\,\pi_\ell(\Frob_q-1)
    \in I_\ell/I_\ell^2. 
$$
Therefore
$$
\det(A_{\{i\}}) = -\frac{h_{n/q}}{h_{n}} \bR{n/q}{n}\,\pi_\ell(\Frob_q-1).
$$
Multiplying \eqref{3} by $h_{n}$ and using \eqref{4} gives
$$
h_{n}\projI{n\ell}{\bR{n}{n\ell}\,v} = h_n\projI{n}{R_{n}}\,\pi_\ell(v) 
   - \sum_{q \mid n_+}h_{n/q}\projI{n\ell}{\bR{n/q}{n}\,v\,\pi_\ell(\Frob_q-1)}.
$$
Since $\bR{n/q}{n} \in I_{n}^r/I_{n}^{r+1}$, we have
$$
\projI{n\ell}{\bR{n/q}{n}\,v\,\pi_\ell(\Frob_q-1)} = \projI{n}{\bR{n/q}{n}\,\pi_n(v)}\,\pi_\ell(\Frob_q-1).
$$
This completes the proof of the proposition.
\end{proof}

If $n \in \N$, recall (Definition \ref{symmdef}) 
that $\symm(n)$ denotes the set of permutations of the primes 
dividing $n_+$, $\symm_1(n) \subset \symm(n)$ is the subset
$$
\symm_1(n) := \{\sigma \in \symm(n) : 
    \text{the primes not fixed by $\sigma$ form a single $\sigma$-orbit}\},
$$
and if $\sigma \in \symm(n)$ then $d_\sigma := \prod_{\sigma(\ell) \ne \ell} \ell$ 
and $\Pi(\sigma) := \prod_{q \mid d_\sigma} \pi_q(\Frob_{\sigma(q)}-1)$.

\begin{thm}
\label{finitepart}
If $n \in\N_p$ and $\ell \mid n_+$, then
$$
\projI{n}{h_n(R_{n})_{\ell,\f}} = 
    -\sum_{\pile{\sigma\in\symm_1(n)}{\sigma(\ell) \ne \ell}} 
        \sign(\sigma) \projI{n/d_\sigma}{h_{n/d_\sigma}(R_{n/d_\sigma})_{\ell}} 
        \Pi(\sigma).
$$
\end{thm}

\begin{proof}
As usual, fix a basis
$\{\lambda_0-\lambda_0^\tau,\ldots,\lambda_r-\lambda_r^\tau\}$ of $X_n^-$ 
with $\ell = \lambda_r\lambda_r^\tau$, and 
an oriented basis $\{\epsilon_0,\ldots,\epsilon_r\}$ of $(1-\tau)\E_n$ as in Lemma \ref{cnr}, 
so that $\{\epsilon_0,\ldots,\epsilon_{r-1}\}$ is an oriented basis of $(1-\tau)\E_{n/\ell}$.  
Then
$$
(R_n)_{\ell,f} = \left|
\begin{array}{ccccccc}
(\epsilon_0)_{\ell,\f} & \cdots & (\epsilon_{r-1})_{\ell,\f} & (\epsilon_r)_{\ell,\f} \\
\grx{\lambda_1}{n}{\epsilon_0} & \cdots & \grx{\lambda_1}{n}{\epsilon_{r-1}} & \grx{\lambda_1}{n}{\epsilon_r} \\
\vdots && \vdots & \vdots\\
\grx{\lambda_r}{n}{\epsilon_0} & \cdots & \grx{\lambda_r}{n}{\epsilon_{r-1}} & \grx{\lambda_r}{n}{\epsilon_r}
\end{array}
\right|.
$$
For each $i$, we have 
$\grx{\lambda_r}{n}{\epsilon_i} = \grx{\lambda_r}{n/\ell}{\epsilon_i} + \grx{\lambda_r}{\ell}{\epsilon_i}$.
If $i < r$, then $\epsilon_i$ is a unit at $\lambda_r$ so 
$\grx{\lambda_r}{n/\ell}{\epsilon_i} = 0$.  Thus
$$
(R_n)_{\ell,f} = \left|
\begin{array}{ccccccc}
(\epsilon_0)_{\ell,\f} & \cdots & (\epsilon_r)_{\ell,\f} \\
\grx{\lambda_1}{n}{\epsilon_0} & \cdots & \grx{\lambda_1}{n}{\epsilon_r} \\
\vdots && \vdots\\
\grx{\lambda_r}{\ell}{\epsilon_0} & \cdots & \grx{\lambda_r}{\ell}{\epsilon_r}
\end{array}
\right| + 
\left|
\begin{array}{ccccccc}
(\epsilon_0)_{\ell,\f} & \cdots & (\epsilon_{r-1})_{\ell,\f} & (\epsilon_r)_{\ell,\f} \\
\grx{\lambda_1}{n}{\epsilon_0} & \cdots & \grx{\lambda_1}{n}{\epsilon_{r-1}} & \grx{\lambda_1}{n}{\epsilon_r} \\
\vdots && \vdots & \vdots\\
0 & \cdots & 0 & \grx{\lambda_r}{n/\ell}{\epsilon_r}
\end{array}
\right|.
$$
The map 
$\epsilon \mapsto \grx{\lambda_r}{\ell}{\epsilon} = [\epsilon,F_{\lambda_r}(\bmu_\ell)/F_{\lambda_r}]-1$ is an isomorphism from 
$(\hat{F}_{\ell,\f}^\times)^- = (\hat{\O}_\ell^\times)^-$ to $(I_\ell/I_\ell^2) \otimes \Zp$, and is 
zero on $(\hat{F}_{\ell,\tr}^\times)^-$ because $\ell$ is a norm in the extension 
$F_{\lambda_r}(\bmu_\ell)/F_{\lambda_r} = \Ql(\bmu_\ell)/\Ql$.  Hence the first 
determinant in the equation above is zero, because the top and bottom rows are 
linearly dependent.  Also, if  $i < r$ then  $\epsilon_i$ is a unit at $\lambda_r$, 
so $(\epsilon_i)_{\ell,\f} = (\epsilon_i)_{\ell}$ and
$$
(R_n)_{\ell,\f} = \left|
\begin{array}{ccccccc}
(\epsilon_0)_{\ell} & \cdots & (\epsilon_{r-1})_{\ell} \\
\grx{\lambda_1}{n}{\epsilon_0} & \cdots & \grx{\lambda_1}{n}{\epsilon_{r-1}} \\
\vdots && \vdots\\
\grx{\lambda_{r-1}}{n}{\epsilon_0} & \cdots & \grx{\lambda_{r-1}}{n}{\epsilon_{r-1}}
\end{array}
\right|\grx{\lambda_r}{n/\ell}{\epsilon_r} 
   = (\bR{n/\ell}{n})_\ell\,\grx{\lambda_r}{n/\ell}{\epsilon_r}.
$$
By Lemma \ref{cnr}, $\grx{\lambda_r}{n/\ell}{\epsilon_r} = -(h_{n/\ell}/h_n)\,\pi_{n/\ell}(\Frob_\ell-1)$.   
Thus
\begin{equation}
\label{1}
h_n\projI{n}{(R_n)_{\ell,\f}} = -h_{n/\ell}\projI{n}{(\bR{n/\ell}{n})_{\ell}\,\pi_{n/\ell}(\Frob_\ell-1)}.
\end{equation}
We can now ``simplify'' \eqref{1} by inductively expanding the right-hand side 
using Proposition \ref{indstep}.  Specifically, expand 
$\projI{n}{(\bR{n/\ell}{n}\,\pi_{n/\ell}(\Frob_\ell-1)}$ 
using Proposition \ref{indstep}(ii).  Then expand each of the resulting 
$\projI{n/\ell}{(\bR{n/(\ell q)}{n/\ell}\,\pi_{n/(q\ell)}(\Frob_q-1)}$ 
using Proposition \ref{indstep}(ii) again.  Continue until no terms $\bR{m/q}{m}$ 
remain.  The resulting sum consists of one term 
$$
(-1)^k \projI{n/(q_1\cdots q_k)}{h_{n/(q_1\cdots q_k)}(R_{n/(q_1\cdots q_k)})_\ell} \prod_{i=1}^k \pi_{q_i}(\Frob_{q_{i+1}}-1)
$$
for each sequence $q_1 = \ell,q_2,\ldots,q_k$ of distinct primes dividing $n_+$ (with 
$q_{k+1} = \ell$).  Identifying this sequence with the $k$-cycle 
$\sigma := (\ell,q_2,\ldots, q_k) \in \symm_1(n)$ gives the formula of the theorem, 
since $\sign(\sigma) = (-1)^{k-1}$.
\end{proof}

\begin{thm}
\label{Rpreks}
The collection $\{h_nR_n : n \in \N_p\}$ is a pre-Kolyvagin system.
\end{thm}

\begin{proof}
We need to check the five properties of Definition \ref{preksdef}.  
Property (i) is Lemma \ref{finlem}, (ii) is Proposition \ref{regreg1lem}, 
(iii) is Proposition \ref{fsprop}, (iv) is Theorem \ref{finitepart} along with 
of Remark \ref{5.9}, and (v) is Proposition \ref{indstep}(i).
\end{proof}

\section{Proof of Theorem \ref{mainthm}}
\label{pfsect}

\begin{proof}[Proof of Theorem \ref{mainthm}]
Fix an odd prime $p$.  
By Theorems \ref{thetapreks} and \ref{Rpreks}, we have pre-Kolyvagin systems
$$
\{2^{-s(n)}\tilde\theta_n : n \in \N_p\}, \quad \{-h_nR_n : n \in \N_p\}.
$$
By Proposition \ref{n=1}, $\tilde\theta'_1 = -h_1R_1$ in $\O_F^\times/\{\pm1\}$.
Hence by Corollary \ref{prekskscor}, 
\begin{equation}
\label{6}
2^{-s(n)}\tilde\theta_n = -h_nR_n 
   \quad\text{in $(F^\times)^- \otimes \In \otimes \Zp$ for every $n \in \N_p$}.
\end{equation}
If $p \mid n \in \N$, then Proposition \ref{crit}(iv) shows that 
$(p-1)\In = 0$.  Therefore $(F^\times)^- \otimes \In \otimes \Zp = 0$ and 
\eqref{6} holds vacuously in this case.  
Since \eqref{6} holds for every $n \in \N$ and every odd prime $p$, 
this completes the proof of Theorem \ref{mainthm}.
\end{proof}

\end{document}